\pdfoutput=1

\documentclass{amsart}

\usepackage{amsfonts,amssymb,amsmath,enumerate,verbatim,mathtools,tikz,bm,mathrsfs,tikz-cd,hyperref,comment}
\usepackage{adjustbox}
\usepackage[all,pdf]{xy}
\hypersetup{
    colorlinks=true,
    linkcolor=blue,
    filecolor=magenta,      
    urlcolor=cyan,
}

\makeatletter
\@namedef{subjclassname@2020}{\textup{2020} Mathematics Subject Classification}
\makeatother

\author[Jian Liu]{Jian Liu}
\address{School of Mathematical Sciences, Shanghai Jiao Tong University, Shanghai 200240, P.R. China. }
\email{liuj231@sjtu.edu.cn}


\renewcommand{\S}{{\mathcal{S}}}

\DeclareMathOperator{\dg}{dg}

\DeclareMathOperator{\h}{H}

\newcommand{\Z}{\mathbb{Z}}

\newcommand{\B}{\mathcal{B}}

\newcommand{\A}{\mathcal{A}}

\newcommand{\T}{\mathcal{T}}

\newcommand{\D}{\mathsf{D}}
\newcommand{\K}{\mathsf{K}}
\newcommand{\C}{\mathsf{C}}

\pgfdeclarelayer{bg}    
\pgfsetlayers{bg,main}

\newcommand{\del}{\partial}

\DeclareMathOperator{\Ima}{Im}
\DeclareMathOperator{\Ker}{Ker}
\DeclareMathOperator{\com}{c}

\DeclareMathOperator{\id}{id}

\DeclareMathOperator{\ac}{ac}

\DeclareMathOperator{\Hom}{Hom}

\DeclareMathOperator{\Y}{Y}

\DeclareMathOperator{\thick}{\mathsf{thick}}
\DeclareMathOperator{\Thick}{\mathsf{Thick}}

\newcommand{\can}{\mathsf{can}}

\newcommand{\per}{\mathsf{per}}

\newcommand{\Mod}{\mathsf{Mod}}
\newcommand{\Inj}{\mathsf{Inj}}
\newcommand{\mo}{\mathsf{mod}}
\newcommand{\Gr}{\mathsf{Gr}}
\newcommand{\gr}{\mathsf{gr}}

\newcommand{\Proj}{\mathsf{Proj}}
\newcommand{\GrInj}{\mathsf{GrInj}}

\newtheorem{theorem}{Theorem}[section]

\newtheorem{proposition}[theorem]{Proposition}

\newtheorem{lemma}[theorem]{Lemma}
\newtheorem{corollary}[theorem]{Corollary}

\theoremstyle{definition}
\newtheorem{example}[theorem]{Example}
\newtheorem{remark}[theorem]{Remark}

\newtheorem{chunk}[theorem]{}

\usepackage[utf8]{inputenc}

\newtheorem*{ack}{Acknowledgements}
\newtheorem{Thm}{Theorem}

\title{Triangulated categories of periodic complexes and orbit categories}
\date{\today}

\begin{document}
\maketitle
\begin{abstract}
   We investigate the triangulated hull of the orbit categories of the perfect derived category and the bounded derived category of a ring concerning the power of the suspension functor. It turns out that the triangulated hull will correspond to the full subcategory of compact objects of certain triangulated categories of periodic complexes. This specializes to Stai and Zhao's result when the ring is a finite dimensional algebra with finite global dimension over a field. As the first application, if $A,B$ are flat algebras over a commutative ring and they are derived equivalent, then the corresponding derived categories of $n$-periodic complexes are triangle equivalent. As the second application, we get the periodic version of the Koszul duality.
\end{abstract}

\section*{Introduction}
 Given an additive category $\A$ and an integer $n\geq 1$, a complex $(X,\del_X)$
over $\A$ is called $n$-periodic if $X^i=X^{i+n}$ and $\del_X^i=\del_X^{i+n}$ for all $i$.
A chain map $f$ between $n$-periodic complexes is a $n$-periodic morphism  if 
$f^i=f^{i+n}$ for all $i$.
A $1$-periodic complex is just a differential object which first appeared in Cartan and Eilenberg's book \cite{CE}. It was systematically studied by Avramov, Buchweitz and Iyengar \cite{ABI}.
Two morphisms $f,g\colon X\rightarrow Y$ of $n$-periodic complexes are called homotopic if there is a homotopy map  $\{\sigma^i\colon X^i\rightarrow Y^{i-1}\}_{i\in \Z}$ from $f$ to $g$ such that $\sigma^i=\sigma^{i+n}$ for all $i$.
Then one can form the homotopy category $\K_n(\A)$ of $n$-periodic complexes and the  derived category $\D_n(\A)$  of $n$-periodic complexes when $\A$ is abelian. They are both triangulated categories;  see \cite{PX1997} or Section \ref{Section 2}. 

Let $R$ be a left noetherian ring. In this article, we will focus on studying the homotopy category $\K_n(R\text{-}\Inj)$ of $n$-periodic complexes of injective $R$-modules and the derived category $\D_n(R\text{-}\Mod)$ of $n$-periodic complexes of  $R$-modules.
For a complex $X$ of $R$-modules, one can associate a $n$-periodic complex $\Delta(X)$:
$$
 \cdots\longrightarrow \coprod_{j\equiv i-1  (\mo ~n)}X^j\longrightarrow\coprod_{j\equiv i (\mo ~n)} X^j\longrightarrow \coprod_{j\equiv i+1 (\mo~ n)}X^j\longrightarrow\cdots;
$$
this process is called \emph{compression} in \cite{ABI}. As expected, the triangulated category of periodic complexes and the classical triangulated category are closely linked by $\Delta$. It is known to experts that the homotopy category $\K(R\text{-}\Inj)$ of complexes of injective $R$-modules and the derived category $\D(R\text{-}\Mod)$ of complexes of $R$-modules are compactly generated; the first one is due to Krause \cite{Krause}. Inspired by this, we prove that $\K_n(R\text{-}\Inj)$ and $\D_n(R\text{-}\Mod)$ are compactly generated; see Theorem \ref{recollement}. Moreover, the full subcategories of compact objects of these two categories are the triangulated hull of certain orbit categories; see Theorem \ref{t1} in the introduction. It is also proved in Theorem \ref{recollement} that the canonical functor $\K_n(R\text{-}\Inj)\rightarrow \D_n(R\text{-}\Mod)$ induces a recollement.

\vspace{10pt}
Let $T\colon \A\rightarrow \A$ be an autoequivalence. Following \cite{Keller05}, the \emph{orbit category} $\A/T$ is defined as follows: it has the same objects as $\A$ and the morphism spaces 
$$
\Hom_{\A/T}(X,Y):=\coprod_{i\in \Z}\Hom_{\A}(X,T^iY).
$$
The composition in $\A/T$ is defined in a natural way. As the name suggests, the objects in the same $T$-orbit are isomorphic. 

If $\T$ is a triangulated category with suspension functor $\Sigma$, there is a natural question: does $\T/\Sigma^n$ inherit a  triangulated structure from $\T$ such that the projection functor $\T\rightarrow \T/\Sigma^n$ is exact?  Let $R$ be a finite dimensional hereditary algebra over a field. Peng and Xiao \cite{PX1997} observed the orbit category $\D^b(R\text{-}\mo)/[2]$ of the bounded derived category of finitely generated $R$-modules, introduced by Happel \cite{Happel87} under the name ``root category", is triangulated. Indeed, they proved that it is equivalent to the homotopy category of 2-periodic complexes of finitely generated projective $R$-modules. This established for the first time a link between the orbit category and the triangulated category of periodic complexes. By making use of this triangulated structure, they constructed the so-called Ringel-Hall Lie algebra determined by $\D^b(R\text{-}\mo)/[2]$ and gave a realization of
all symmetrizable Kac-Moody Lie algebras; see \cite{PX2000}.


However, Neeman found the answer to the above question is negative; see discussions in  \cite{Keller05}. Inspired by questions from the cluster category in \cite{BMRRT}, Keller \cite{Keller05} constructed the \emph{triangulated hull} of certain orbit categories. As an application, he proved the cluster category in \cite{BMRRT} is triangulated.

Keller's construction is an abstract embedding of the certain orbit category into the triangulated hull. It will be nice to know the precise triangulated hull of a given orbit category. If $R$ is a finite dimensional algebra with finite global dimension over a field, it was independently proved by Stai \cite{Stai} and Zhao \cite{Zhao} that $\D^b(R\text{-}\mo)/[n]$ embeds into its triangulated hull $\D_n(R\text{-}\mo)$, where $R\text{-}\mo$ is the category of finitely generated $R$-modules.
We are motivated by the natural question: \emph{what is the triangulated hull of $\D^b(R\text{-}\mo)/[n]$ without these assumptions of $R$?}

Our first result Theorem \ref{t1} answers this question.  It extends Stai and Zhao's result; see Corollary \ref{SZ}. 

Recall the perfect derived category $\per (R)$ is the full subcategory of $\D(R\text{-}\Mod)$ formed by complexes that are quasi-isomorphic to bounded complexes of finitely generated projective $R$-modules. Recall the embedding $\mathsf{i}\colon \D^b(R\text{-}\mo)\rightarrow \K(R\text{-}\Inj)$ induced by taking injective resolution; see \ref{compact object in K(Inj R)}. For a triangulated category $\T$ with corpoducts, an object $X$ is called compact if $\Hom_{\T}(X,-)$ preserves coproducts. We let $\T^{\com}$ denote the full subcategory of $\T$ formed by compact objects and note that $\T^{\com}$ is a thick subcategory of $\T$.

\vspace{3pt}
\begin{Thm}\label{t1}
(see \ref{embedding in hull}) Let $R$ be a left noetherian ring. 
Induced by the compression of complexes, the functors
$$
\overline{\Delta\circ \mathsf{i}}\colon \D^b(R\text{-}\mo)/[n]\longrightarrow \K_n(R\text{-}\Inj )^{\com}\text{ and }~\overline{\Delta}\colon\per(R)/[n]\longrightarrow \D_n(R\text{-}\Mod )^{\com}
$$
are the embedding of the orbit categories into their triangulated hull.
\end{Thm}

In order to prove Theorem \ref{t1}, we realize $\K_n(R\text{-}\Inj)$ and $\D_n(R\text{-}\Mod)$ as derived categories of orbit categories of certain dg categories; see Theorem \ref{main result}.

\vspace{10pt}
Motivated by Theorem \ref{t1}, we compare the triangle equivalences $\D(A\text{-}\Mod)\simeq \D(B\text{-}\Mod)$ and $\D_n(A\text{-}\Mod )\simeq \D_n(B\text{-}\Mod)$ for two rings $A, B$ in Section \ref{Section 4}. Two rings that satisfy the first equivalence are called derived equivalent.
In general, whether two rings are derived equivalent is difficult to grasp. Therefore, it is important to investigate the invariant under the derived equivalence. By introducing the tilting complex, Rickard \cite{Rickard1989} established the derived Morita theory of rings. After that, Keller \cite{Keller94} generalized Rickard's derived Morita theory through the language of differential graded categories. 

 It turns out that the above two equivalences are closely related; see Proposition \ref{derived equivalence}.   In particular, combine Proposition \ref{derived equivalence} with \cite{Keller94}, we get the following result which extends a result of Zhao \cite{Zhao}; see Corollary \ref{finite global}. 

\vspace{3pt}
\begin{Thm}\label{t2} (see \ref{implication})
Let $k$ be a commutative ring and $A,B$ be flat $k$-algebras. If $A$ and $B$ are derived equivalent, then $\D_n(A\text{-}\Mod)$ and $\D_n(B\text{-}\Mod)$ are equivalent as triangulated categories.
\end{Thm}

\vspace{3pt}
In the last section, we give the \emph{periodic version of the Koszul duality} (Theorem \ref{t3}). Its proof relies on the classical Koszul duality and the studies in previous sections.

Let $k$ be a field and $S=k[x_1,\ldots,x_c]$, where $\mathrm{deg}(x_i)=1$. Denote by $\Lambda$ the graded exterior algebra over $k$ on variables $\xi_1,\ldots,\xi_c$ of degree $-1$. Bernstein, Gel'fand and Gel'fand \cite{BGG} established the  triangle equivalence 
$
\D^b(\Lambda\text{-}\gr)\simeq \D^b(S\text{-}\gr)
$
between the bounded derived category of finitely generated graded modules. This is known as the BGG correspondence.

The BGG correspondence can be lifted to the compact completions. That is, there is a triangle equivalence
$
\K(\Lambda\text{-}\GrInj)\simeq\D(S\text{-}\Gr)
$
(see \cite{Krause} or \ref{classical BGG}), where $A\text{-}\Gr$ is the category of graded modules over the graded algebra $A$ and $\Lambda\text{-}\GrInj$ is the full subcategory of $\Lambda\text{-}\Gr$ formed by injective objects. The corresponding equivalence will be called the Koszul duality. The Koszul duality phenomenon has played an important role in representation theory. For instance,  the DG version of the Koszul duality was used by Benson, Iyengar and Krause \cite{BIK} to stratify the modular representation theory of finite groups.

\vspace{3pt}
\begin{Thm}\label{t3}(see \ref{Koszul duality} and \ref{version of BGG})
There is a triangle equivalence 
$$ 
\K_n(\Lambda\text{-}\GrInj )\stackrel{\sim}\longrightarrow \D_n(S\text{-} \Gr ).
$$
Moreover, this equivalence induces an embedding
$$
\D^b(\Lambda\text{-}\gr)/[n]\longrightarrow \D_n(S\text{-}\gr)
$$ 
of the orbit category into its triangulated hull.
\end{Thm}
\vspace{3pt}
\begin{ack}
Part of the work was done during the author's visit to the University of Utah. The author would like to thank Srikanth Iyengar for his kind hospitality and the China Scholarship Council for their financial support. Special thanks to Benjamin Briggs for providing a similar project related to Section \ref{Section 5} which makes the article possible. The author thanks Xiao-Wu Chen, Srikanth Iyengar, Janina Letz, and Josh Pollitz for their discussions on this work.
\end{ack}

\section{Notations and Preliminaries}
 Throughout the article, $R$ is a left noetherian ring. $R\text{-}\Mod$ (resp. $R\text{-}\mo$) will be the category of left (resp. finitely generated left) $R$-modules. The full subcategory of $R\text{-}\Mod$ consisting all projective (resp. injective) $R$-modules is denoted by $R\text{-}\Proj$ (resp. $R\text{-}\Inj $).
For an additive category $\A$,  $\C(\A)$ will be the category of complexes over $\A$ with suspension functor $[l]$ ($X[l]^i:=X^{i+l},~ \del_{X[l]}^i:=(-1)^l\del_X^{i+l}$). Denote by $\K(\A)$ the homotopy category of complexes over $\A$. When $\A$ is abelian, let $\D(\A)$ denote the derived category of complexes over  $\A$.

A complex of $R$-modules is \emph{perfect} provided that it is quasi-isomorphic to a bounded complex of finitely generated projective $R$-modules. $\per (R)$ will be the full subcategory of $\D(R\text{-}\Mod)$ consisting of all perfect complexes. 

\begin{chunk}
\textbf{Thick subcategories and localizing subcategories.} Let $\T$ be triangulated category and $\mathcal{C}$ be a triangulated subcategory of $\T$. We say $\mathcal{C}$ is \emph{thick} (resp. \emph{localizing}) if it is closed under direct summands (resp. coproducts).  
For a set $S$ of object in $\T$, we let $\thick_{\T}(S)$ denote the smallest thick subcategories of $\T$ containing $S$. This can be realized as the intersection of all thick subcategories of $\T$ containing $S$; it has an inductive construction (see \cite[2.2.4]{ABIM}).

If $\T$ has coproducts, then a technique of Eilenberg's swindle implies that any localizing subcategory is thick.
\end{chunk}
It is well-known that $\per (R)=\thick_{\D(R\text{-}\Mod)}(R)$ is thick; see \cite[Lemma 1.2.1]{Buchweitz}.
\begin{chunk}
Let $F\colon \T\rightarrow\T^\prime$ be an exact functor between triangulated categories. Then the \emph{kernel} of $F$ defined by
$$
\Ker F:=\{X\in \T\mid F(X)\cong 0\}
$$
is a thick subcategory of $\T$. When the functor $F$ is full, the \emph{essential image} of $F$ defined by 
$$
\Ima F:=\{Y\in \T^\prime\mid Y\cong F(X) \text{ for some } X\in \T\}
$$
is a triangulated subcategory of $\T^\prime$.
\end{chunk}
\begin{chunk}\label{def of recollement}
\textbf{Recollement.} Following Beilinson, Bernstein and Deligne \cite{BBD}, we call a diagram 
\[\begin{tikzcd}
\T^\prime\arrow[rr,"{i_*}"]&&\T \arrow[ll,bend right,"{i^*}"swap]\arrow[ll,bend left,"{i^!}"swap]\arrow[rr,"{j^\ast}"]&& \T^{\prime\prime}\arrow[ll, bend right, "{j_!}"swap]\arrow[ll,bend left, "{j_\ast}"swap]
\end{tikzcd}
\]
of triangulated categories and exact functors \emph{a recollement} if the following conditions are satisfied.

(1) $(i^\ast,i_\ast),~(i_\ast,i^!),~(j_!,j^\ast)$ and $(j^\ast,j_\ast)$ are adjoint pairs.

(2) $i_\ast,~j_!$ and $j_\ast$ are fully faithful.

(3) $\Ima i_\ast=\mathrm{Ker} j^\ast$, that is, $j^\ast(X)=0$ if and only if $X\cong i_\ast (Y)$ for some $Y\in \T^\prime.$
\end{chunk}
Next, we record a useful result; its proof can refer \cite[Section 3]{Krause}.
\begin{chunk}\label{inclusion has right iff quo has}
Let $\S$ be a thick subcategory of $\T$. Then the inclusion functor $\S\rightarrow \T$ has a right (resp. left) adjoint if and only if the quotient functor $Q\colon\T\rightarrow \T/\S$ has a right (resp. left) adjoint. In this case,  the right (resp. left) adjoint of the quotient functor is fully faithful. The sequence $\S\xrightarrow {\mathrm{inc}} \T\xrightarrow Q\T/\S$ in this case is called a \emph{localization sequence} (resp. \emph{colocalization sequence}).

Assume the sequence $\S\xrightarrow {\mathrm{inc}} \T\xrightarrow Q\T/\S$ is a localization sequence. Denote by $\pi$ (resp. $\iota$) the right adjoint of the functor $\mathrm{inc}$ (resp. $Q$). Then the sequence 
$
\T/\S\xrightarrow{\iota}\T\xrightarrow{\pi} \S
$
is a colocalization sequence. In particular, $\pi$ induces a triangle equivalence 
$$
\T/\Ima\iota\stackrel{\sim} \longrightarrow\S.
$$
\end{chunk}
Note that a sequence $\T^\prime\rightarrow \T\rightarrow \T^{\prime\prime}$ induces a recollement as \ref{def of recollement} if and only if the sequence is both a localization sequence and a colocalization sequence.
\begin{chunk}\label{Neeman's description of compact objects}
\textbf{Compactly generated triangulated categories.} Let $\T$ be a triangulated category with coproducts. An object $X\in \T$ is called \emph{compact} provided that the Hom functor $\Hom_{\T}(X,-)$ commutes with coproducts. That is, for any class of objects $Y_i(i\in I)$ in $\T$, the canonical map 
$$
\can\colon\coprod_{i\in I}\Hom_{\T}(X,Y_i)\longrightarrow  \Hom_{\T}(X,\coprod_{i\in I}Y_i)
$$
is isomorphic. We let $\T^{\com}$ denote the full subcategory of $\T$ formed by compact objects in $\T$. It is not hard to show that $\T^{\com}$ is a thick subcategory of $\T$.

$\T$ is said to be \emph{compactly generated} if there exists a set $S$ of compact objects such that any object $Y$ satisfying $\Hom_{\T}(X,Y[i])=0$ for all $X\in S$ and $i\in \Z$ is a zero object; the condition is equivalent to $\T$ is equal to the smallest localizing subcategory containing $S$ (see \cite[Lemma 3.2]{Neeman96}). In this case, $\T^{\com}=\thick_{\T}(S)$; see \cite[Lemma 2.2]{Neeman92}. For instance, $\D(R\text{-}\Mod)$ is compactly generated by the compact object $R$.
\end{chunk}

A set $S$ of objects in $\T$ is called a \emph{compact generating set} provided that $S\subseteq \T^{\com}$ and $\T$ is compactly generated by $S$.
The following result is well-known. For its proof, we refer the reader to \cite[Lemma 4.5]{BIK}; compare \cite[Lemma 1]{Beilinson} and \cite[Lemma 4.2]{Keller94}.
\begin{lemma}\label{test equivalence}
Let $F\colon \T\rightarrow \T^\prime$ be an exact functor between compactly generated triangulated categories. Assume $F$ preserves coproducts and $S\subseteq \T^{\com}$ is a compact generating set. Then $F$ is fully faithful if and only if the induced maps
$$
\Hom_\T(X,\Sigma^i(Y))\longrightarrow \Hom_{\T^\prime}(FX,F\Sigma^i(Y))
$$
are isomorphic for all $X,Y\in S$ and $i\in \Z$. In this case, $F$ is dense if and only if $\Ima F$ contains a compact generating set of $\T^\prime.$
\end{lemma}
\begin{chunk}
\textbf{DG categories and dg functors.} An additive category $\A$ is called a \emph{dg category} provided that for each $X,Y\in \A$, the morphism space $\Hom_{\A}(X,Y)$ is a complex and the compositions
$$
\Hom_{\A}(Y,Z)\otimes_{\Z}\Hom_{\A}(X,Y)\longrightarrow \Hom_{\A}(X,Z)
$$
are chain maps. 
An additive functor $F\colon \A\rightarrow \B$ is called a \emph{dg functor} provided that $F$ commutes with the differential.
\end{chunk}
Let $\A$ be an additive category. Denote by $\C_{\dg}(\A)$ the dg category of  complexes over $\A$ whose mophism spaces are Hom complex defined by 
$$
\Hom_{\A}(X,Y)^i=\prod_{p\in \Z}\Hom_{\A}(X^p,Y^{p+i})
$$
with differential $\del(f)=\del_Y\circ f+(-1)^{\mid f \mid}f\circ \del_X$. 

The \emph{homotopy category} $\h^0(\A)$ of $\A$ is defined to be the category with same objects of $\A$ and the morphism spaces are the zeroth cohomology of the corresponding Hom complexes in $\A$. Observe that 
$$
\h^0(\C_{\dg}(\A))=\K(\A).
$$
\begin{chunk}
\textbf{Derived categories of dg categories.}
We briefly discuss the derived category of a dg category. See \cite{Keller94} for more details.

Let $\A$ be a small dg category. A  \emph{dg module} over $\A$ is a dg functor
$$
M\colon \A^{\mathrm{op}}\longrightarrow \C_{\dg}(\Z\text{-}\Mod).
$$
Then the category of dg $\A$-modules, denoted $\Mod_{\dg}(\A)$,  is still a dg category; see \cite[Section 1.2]{Keller94}. Its homotopy category $\h^0(\Mod_{\dg}(\A))$ is a triangulated category; see \cite[Lemma 2.2]{Keller94}. A dg $\A$-module is called \emph{acyclic} if $M(X)$ is acyclic for each object $X\in \A$. The \emph{derived category} of $\A$ is defined to be the Verdier quotient of $\h^0(\Mod_{\dg}(\A))$ by its full subcategory of acyclic dg $\A$-modules.
\end{chunk}
We have the Yoneda embedding
$$
\mathrm{Y}\colon \h^0(\A)\longrightarrow \D(\A), ~X\mapsto \Hom_{\A}(-,X).
$$
It is well-known that $\D(\A)$ is compactly generated by the image of $\mathrm{Y}$; see \cite[4.2]{Keller94}.
\begin{chunk}\label{pretriangulated category}
\textbf{Pretriangulated category.}
Keep the notation as above. The dg category $\A$ is called \emph{pretriangulated} if $\Ima Y$ is a triangulated category. In this case, $\h^0(\A)$ inherits a natural triangulated structure and there is (up to direct summands) a triangle equivalence 
$$
\h^0(\A)\stackrel{\sim}\longrightarrow \D(\A)^{\com}.
$$
\end{chunk}
\begin{chunk}\label{def of dg enhancement}
\textbf{DG enhancement.}
Let $\T$  be a triangulated category and $\A$ be a dg category. $\A$ is said to be a \emph{dg enhancement} of $\T$ provided that $\A$ is pretriangulated and $\T$ is triangle equivalent to $\h^0(\A)$ endowed with the natural triangulated structure (see \ref{pretriangulated category}). In this case, any triangulated subcategory $\S$ of $\T$ has a dg enhancement. Indeed, denote by $\A^\prime$ the full dg subcategory of $\A$ consisting of objects in the essential image of $\S$. Then $\A^\prime$ is a dg enhancement of $\S$; see \cite[Section 2.2]{KY}.
\end{chunk}
Let $\A$ be an additive category. Then $\C_{\dg}(\A)$ is pretriangulated and is a dg enhancement of $\K(\A)$.
\begin{example}\label{example of dg enhancement}
 By above, $\C_{\dg}(R\text{-}\Mod)$ (resp. $\C_{\dg}(R\text{-}\Inj)$) is a dg enhancement of $\K(R\text{-}\Mod)$ (resp. $\K(R\text{-}\Inj)$). Denote by $\per_{\dg}(R)$ the full dg subcategory of $\C_{\dg}(R\text{-}\Mod )$ consisting of all perfect complexes. Then $\per_{\dg}(R)$ is a dg enhancement of $\per (R)$.

Next we give an example that is used in Section \ref{Section 3}. We write $\C^{+,f}_{\dg}(R\text{-}\Inj)$ to be the full subcategory of $\C_{\dg}(R\text{-}\Inj)$ formed by bounded below complexes whose total cohomology are finitely generated $R$-modules. Induced by taking injective resolution,  there exists a triangle equivalence
 $$
 \D^b(R\text{-}\mo)\stackrel{\sim}\longrightarrow \h^0(\C^{+,f}_{\dg}(R\text{-}\Inj)).
 $$
\end{example}

\section{Triangulated categories of periodic complexes}\label{Section 2}
Throughout the article, $n\geq 1$ is an integer.
In this section, we investigate periodic complexes. Remarkably, there exists an adjoint pair between the classical triangulated category and the corresponding triangulated category of periodic complexes. It is proved that many  properties of the latter can be determined by the former. The main result in this section is Theorem \ref{recollement}.

Let $\A$ be an additive category, denote by $\C_n(\A)$ the category of $n$-periodic complexes over $\A$ whose morphism spaces are $n$-periodic morphisms; see the introduction.
For each $l\in \Z$, there is a canonical suspension functor $[l]$  on $\C_n(\A)$ which maps $X$ to $X[l]$ ($X[l]^i:=X^{i+l}$, $\del^i_{X[l]}:=(-1)^l\del^{i+l}_X)$ and acts trivially on morphisms.
\begin{chunk}
\textbf{Homotopy category of $n$-periodic complexes.} Let $\A$ be an additive category and $X,Y\in \C_n(\A)$. Two morphisms $f,g\colon X\rightarrow Y$ are called \emph{homotopic} if there exists a sequence $\{\sigma^i\colon X^i\rightarrow Y^{i-1}\}_{i\in \Z}$ of morphisms over $\A$ such that $f^i-g^i=\sigma^{i+1}\circ \del_X^i+\del_Y^{i-1}\circ \sigma^i$ and $\sigma^i=\sigma^{i+n}$ for all $i\in \Z$.

The \emph{homotopy category} of $n$-periodic complexes over $\A$, denoted $\K_n(\A)$, is defined by identifying homotopy in $\C_n(\A)$. It is a triangulated category with suspension functor [1]; see \cite[Section 7]{PX1997}.

Let $f\colon X\rightarrow Y$ be a morphism in $\C_n(\A)$. The \emph{mapping cone} $C(f)$ of $f$  is
$$
C(f)^i:=X^{i+1}\coprod Y^i,~\del_{C(f)}^i:=
{\begin{pmatrix}
-\del_X^{i+1} & 0\\
f^{i+1}& \del_Y^i
\end{pmatrix}.
}
$$
In $\K_n(\A)$, $f$ can be embedded in a canonical exact triangle
$$
\xymatrix{
X\ar[r]^-f & Y\ar[r]^-{\begin{pmatrix}\begin{smallmatrix}0\\1
\end{smallmatrix}\end{pmatrix}}& C(f)\ar[r]^-{(1~0)}& X[1]
}
$$
\end{chunk}

As Peng and Xiao \cite[7.1]{PX1997} mentioned, $\C_n(\A)$ is a subcategory of $\C(\A)$ (usually not full) and $\K_n(\A)$ is usually not a subcategory of $\K(\A)$.
\begin{chunk}
\textbf{Derived category of $n$-periodic complexes.} Let $\A$ be an abelian category. A $n$-periodic complex $X$ is called \emph{acylcic} if it is acyclic as complex, i.e. $\h^i(X):=\Ker(\del^i_X)/\Ima(\del^{i-1}_X)=0$ for all $i\in \Z$. The \emph{derived category} of $n$-periodic complexes over $\A$, denoted $\D_n(\A)$, is the Verdier quotient category of $\K_n(\A)$ by its full subcategory of acyclic $n$-periodic complexes.
\end{chunk}
Following the definition of the compression for the case $n=1$ in \cite[1.3]{ABI}, we define the compression for arbitrary $n\geq 1$; see also \cite{Stai}.
\begin{chunk}\label{counit}
\textbf{Compression.}
Let $\A$ be an additive category with coproducts.  For a complex $X\in \C(\A)$
$$
\xymatrix{
   \cdots\ar[r]& X^{i-1}\ar[r]^-{\del_X^{i-1}}& X^i\ar[r]^-{\del_X^i}& X^{i+1}\ar[r]& \cdots,
   }
$$
The \emph{compression} $\Delta(X)$ of $X$ is defined by
$$
 \cdots\longrightarrow \coprod_{j\equiv i-1  (\mo ~n)}X^j\longrightarrow\coprod_{j\equiv i (\mo ~n)} X^j\longrightarrow \coprod_{j\equiv i+1 (\mo~ n)}X^j\longrightarrow\cdots
$$
with the natural differential induced by the differential of $X$, where the $i$-th component of $\Delta(X)$ is $\coprod_{j\equiv i(\mo~ n)}X^j.$

This gives an additive functor $\Delta:\C(\A)\rightarrow \C_n(\A)$.
\end{chunk}
Clearly, there is a natural exact functor $\nabla\colon\C_n(\A)\rightarrow \C(\A )$ which maps a periodic complex to itself. 
We observe that $(\Delta,\nabla)$ is an adjoint pair.  For each $X$ in $\C(\A)$, it is not hard to see there is an isomorphism $\nabla\Delta(X)\cong \coprod_{i\in \Z}X[ni]$. Moreover, the unit $\eta_X\colon X\rightarrow \nabla\Delta(X)$ corresponding to the adjoint pair is the composition 
$$
X\stackrel{\can }\longrightarrow\coprod_{i\in \Z}X[ni]\cong \nabla\Delta(X).
$$
\begin{chunk}\label{adjoint}
Keep the notation as \ref{counit}. One can check directly that  $\Delta$ and $\nabla$ preserves homotopy, suspensions and mapping cones. Hence they induce an adjoint pair of exact functors between the homotopy categories
\[\begin{tikzcd}
\K(\A) \arrow[r,shift left=0.8ex,"\Delta"]&\K_n(\A)\arrow[l,shift left=0.8ex,"\nabla"].
\end{tikzcd}\]
If $\A$ is also abelian, we observe $\nabla$ induces an exact functor $\nabla\colon \D_n(\A)\rightarrow \D(\A)$. If further $\A$ is an AB4 category (i.e. an abelian category with coproducts and the coproduct is an exact functor), then $\Delta$ preserves acyclic objects. Thus $\Delta$ naturally induces an exact functor $\Delta\colon \D(\A)\rightarrow \D_n(\A) $. Moreover, $(\Delta,\nabla)$ is an adjoint between the derived categories; see \cite[Lemma 1]{Orlov2009}.
\end{chunk}

\begin{chunk}\label{preserve coproduct}
If $\A$ is an additive category with coproducts, then it can be checked directly that both  $\K(\A)$ and $\K_n(\A)$ have coproducts. If $\A$ is an AB4 category,  then both $\D(\A)$ and $\D_n(\A)$ have coproducts; see 
\cite[Proposition 3.5.1]{krause2010}.

In addition, in these cases, the degree-wise coproduct of objects in $\K(A)$ (resp. $\K_n(\A)$, $\D(A)$, $\D_n(\A)$) is the categorical coproduct.
\end{chunk}
Similar results of \ref{preserve coproduct} hold when we replace the coproduct by the product and replace an AB4 category by an AB4* category (i.e. an abelian category with products and the product is an exact functor).
\begin{lemma}\label{description of compact object}
(1) Let $\A$ be an additive category with coproducts and $X$ be an object in $\K(\A)$ .  Then $X$ is compact in $\K(\A)$  if and only if $\Delta(X)$ is compact in $\K_n(\A)$.

(2) Let $\A$ be an AB4 category and $X$ be an object in $\D(\A)$ .  Then $X$ is compact in $\D(\A)$  if and only if $\Delta(X)$ is compact in $\D_n(\A)$.
\end{lemma}
\begin{proof}
We prove (1). The proof of (2) is similar. First, assume $X$ is compact in $\K(\A)$. Since $(\Delta,\nabla)$ is an adjoint pair and $\nabla$ preserves coproducts (c.f. \ref{preserve coproduct}), $\Delta$ preserves compact objects; see \cite[Theorem 5.1]{Neeman96}. Thus $\Delta(X)$ is compact in $\K_n(\A)$.

For the converse, assume $\Delta(X)$ is compact in $\K_n(\A).$  For a class of objects $Y_i (i\in I)$ in $\K(\A)$, consider the commutative diagram
$$
\xymatrix{
\displaystyle\coprod_{i\in I}\Hom_{\K(\A)}(X,Y_i)\ar[d]_-{\coprod_{i\in I}(\eta_{Y_i})_\ast}\ar[r]^-{\can}& \Hom_{\K(\A)}(X,\displaystyle\coprod_{i\in I}Y_i)\ar[d]^-{(\eta_{\coprod_{i\in I}Y_i})_\ast}\\
\displaystyle\coprod_{i\in I}\Hom_{\K(\A)}(X,\nabla\Delta(Y_i))& \Hom_{\K(\A)}(X,\nabla\Delta(\displaystyle\coprod_{i\in I}Y_i))\\
\displaystyle\coprod_{i\in I}\Hom_{\K_n(\A)}(\Delta(X),
\Delta(Y_i))\ar[u]^-{\cong}\ar[r]^-{\cong} & \Hom_{\K_n(A)}(\Delta(X),\Delta(\displaystyle\coprod_{i\in I}Y_i))\ar[u]_-{\cong}
}
$$
where the vertical isomorphisms are induced by adjoint pair $(\Delta, \nabla)$ and the horizontal one is based on the assumption. Since the unit $\eta_M\colon M\rightarrow \nabla\Delta(M)$ is split injection for each $M\in \K(\A)$ (see \ref{counit}), we conclude that $\can$ is isomorphism.
\end{proof}
\begin{example}\label{finite projective flag}
Let $n=1$. Following  \cite{ABI}, a differential $R$-module $(P,\delta_P)$ admits a \emph{finite projective flag} if $P=P_0\coprod P_1\coprod \cdots\coprod P_l$ and $\delta_P$ is of the form
$$
\begin{pmatrix}
0&   \del_{1,0}&   \del_{2,0}&   \cdots&   \del_{l-1,0}&     \del_{l,0}\\
0&           0&   \del_{2,1}&   \cdots&   \del_{l-1,1}&     \del_{l,1}\\
0&           0&            0&   \cdots&   \del_{l-1,2}&     \del_{l,2}\\
\vdots& \vdots&       \vdots&   \ddots&         \vdots&         \vdots\\
0&           0&            0&   \cdots&              0&   \del_{l,l-1}\\
0&           0&            0&   \cdots&              0&              0\\
\end{pmatrix}
$$
where each $P_i$ is finitely generated projective $R$-module. Set $F^i=(P_0\coprod\cdots\coprod P_i,\delta_P)$ ($0\leq i\leq l$). These are differential submodules of $(P,\delta_P)$. It follows that $(P,\delta_P)$ has a filtration
$$
F^0\subseteq F^1\subseteq \cdots \subseteq F^l=(P,\delta_P)
$$
such that $F^i/F^{i-1}\cong (P_i,0)$ for each $i$. Since $\Delta(R)=(R,0)\in \D_1(R\text{-}\Mod)^{\com}$ (see Lemma \ref{description of compact object}), the differential modules that admit finite projective flags are compact objects in $\D_1(R\text{-}\Mod)$.
\end{example}
If $\A$ is an abelian category, then an object $X$ in $\D_n(\A)$ is zero  if and only if $\nabla(X)$ is zero in $\D(\A)$. Similar result holds in the homotopy category; see next lemma.
For the case of $n=1$ and $\A$ is the category of modules over a ring, it was obtained by Avramov, Buchweitz and Iyengar \cite[Proposition 1.8]{ABI}.
\begin{lemma}\label{zero in homotopy category}
Let $\mathcal A$ be an additive category and $X$ be an object in  $\K_n(\mathcal{A})$. Then $X$ is zero in $\K_n(\mathcal{A})$ if and only if $\nabla(X)$ is zero in $\K(\mathcal A)$.
\end{lemma}
\begin{proof}
The forward direction is trivial. For the converse, assume $\nabla(X)$ is zero in $\K(\A)$. Then there exists $s^i\in \Hom_{\A}(X^i,X^{i-1})$ for all $i\in \Z$ such that 
\begin{equation}\label{equation}
    \id_{X^i}=s^{i+1}\circ \del_X^i+\del_X^{i-1}\circ s^i.
\end{equation} 
We define $\sigma^i\colon X^i\rightarrow X^{i-1}$ as follows
$$
\sigma^i=\begin{cases}
s^n\circ \del_X^{-1}\circ s^0, &\text{if } i\equiv 0(\mo ~n)\\
s^j& \text{if } i\equiv j(\mo~n) \text{ and } 1\leq j\leq n-1
\end{cases}.
$$
Our aim is to show that this gives the homotopy map from $\id_X$ to $0$ in $\K_n(\A)$. Due to the choice of $\sigma^i$, it remains to check that $\id_{X^0}=\sigma^1\circ \del_X^0+\del_X^{-1}\circ \sigma^0$ and $\id_{X^{n-1}}=\sigma^0\circ \del_X^{n-1}+\del_X^{n-2}\circ \sigma^{n-1}$.
Indeed, these are direct consequences of  (\ref{equation}). Thus $X$ is zero in $\K_n(\A)$.
\end{proof}
\begin{lemma}\label{compactly generated}
(1) Let $\mathcal A$ be an additive category with coproducts. If $\K(\mathcal A)$ is compactly generated, then so is $\K_n(\A)$ and it's compactly generated by the image of $\K(\A)^{\com}$ under the compression functor.

(2) Let $\mathcal A$ be an AB4 category. If $\D(\mathcal A)$ is compactly generated, then so is $\D_n(\A)$ and it's compactly generated by the image of $\D(\A)^{\com}$ under the compression functor.
\end{lemma}
\begin{proof}
We prove (1). The proof of (2) is similar.
Suppose $\K(\mathcal A)$ is compactly generated. Lemma \ref{description of compact object} yields   $\Delta(\K(\mathcal A)^{\com})\subseteq \K_n( \mathcal A)^{\com}$.
Let $X\in \K_n( \mathcal A)$ and $$\Hom_{\K_n( \mathcal A)}(\Delta(\K(\mathcal A)^{\com}),X)=0.$$ In order to show $\K_n( \mathcal A)$ is compactly generated by $\Delta(\K(\A)^{\com})$, we need to prove $X=0$ in $\K_n(\mathcal A)$. By adjoint we have
$$
\Hom_{\K(\mathcal A)}(\K(\mathcal A)^{\com},\nabla(X))=0.
$$
Then the assumption implies  $\nabla(X)=0$. It follows immediately from Lemma \ref{zero in homotopy category} that $X=0$ in $\K_n(\mathcal A)$. As required.
\end{proof}

The following result is due to Neeman; see \cite[Theorem 4.1]{Neeman96} and \cite[Theorem 8.6.1]{Neeman01}.
\begin{chunk}\label{Bousfield}
Let $\S$ be a compactly generated triangulated category and $F\colon \S\rightarrow \T$ be an exact functor between triangulated categories. Then 

(1) $F$ has a right adjoint if and only if $F$ preserves coproducts.

(2) $F$ has a left adjoint if and only if $F$ preserves  products.
\end{chunk}

As we assume $R$ is a left noetherian ring,  the direct sum of injective $R$-modules is still injective; see \cite[Theorem 3.1.17]{EJ}. Hence $R\text{-}\Inj$ is an additive category with coproducts. 
\begin{chunk}\label{Krause's recollement}
Krause proved that $\K(R\text{-}\Inj)$ and the full subcategory of $\K(R\text{-}\Inj)$ formed by acyclic complexes (denoted $\K^{\ac}(R\text{-}\Inj)$) are compactly generated triangulated categories; see \cite[Proposition 2.3, Corollary 5.4]{Krause}. Moreover, Krause \cite[Corollary 4.3]{Krause} observed that the canonical sequence
$$
\K^{\ac}(R\text{-}\Inj) \stackrel{\mathrm{inc}}\longrightarrow\K(R\text{-}\Inj )\stackrel{Q}\longrightarrow \D(R\text{-}\Mod)
$$
induces a recollement; the definition of recollement is recalled in \ref{def of recollement}.
\end{chunk}
Let $\K_n^{\ac}(R\text{-}\Inj)$ denote the full subcategory of $\K_n(R\text{-}\Inj)$ formed by acyclic complexes. This is a localizing subcategory of $\K(R\text{-}\Inj)$.
Next, we give the periodic version of Krause's result in \ref{Krause's recollement}.
\begin{theorem}\label{recollement}
Let $R$ be a left noetherian ring. Then

(1) $\K^{\ac}_n(R\text{-}\Inj), \K_n(R\text{-}\Inj )$ and $\D_n(R\text{-}\Mod)$ are compactly generated triangulated categories.

(2)
The sequence 
$$
\K^{\ac}_n(R\text{-}\Inj) \stackrel{\mathrm{inc}}\longrightarrow\K_n(R\text{-}\Inj )\stackrel{Q}\longrightarrow \D_n(R\text{-}\Mod)
$$
induces a recollement

\adjustbox{scale=0.8,center}{
$$
\begin{tikzcd}
\K^{\ac}_n(R\text{-}\Inj )\arrow[rr,"\mathrm{inc}"]&&\K_n(R\text{-}\Inj ) \arrow[ll,bend right]\arrow[ll,bend left]\arrow[rr,"Q"]&& \D_n(R\text{-}\Mod ).\arrow[ll, bend right]\arrow[ll,bend left]
\end{tikzcd}
$$
}
\end{theorem}
\begin{proof}
(1) Combine with \ref{Krause's recollement}, it follows immediately from Lemma \ref{compactly generated} that $\K_n(R\text{-}\Inj)$ and $\D_n(R\text{-}\Mod)$ are compactly generated. Also, with the same proof of Lemma \ref{compactly generated}, $\K_n^{\ac}(R\text{-}\Inj)$ is compactly generated.

(2) Next, we borrow Krause's idea in the proof of \cite[Corollary 4.3]{Krause}. Since $Q$ preserves coproducts and products (c.f. \ref{preserve coproduct}),  $\mathrm{Q}$ has both a left adjoint and a right adjoint by (1) and \ref{Bousfield}.
Combine with \ref{inclusion has right iff quo has}, it remains to show $Q$ induces a triangle equivalence $\K_n(R\text{-}\Inj )/\K^{\ac}_n(R\text{-}\Inj )\cong \D_n(R\text{-}\Mod )$. Again \ref{inclusion has right iff quo has} yields this is equivalent to show the right adjoint of $Q$  is fully faithful.

Denote by $Q_\rho$ the right adjoint of $Q$.  It is clear that the inclusion functor $J\colon \K_n(R\text{-}\Inj )\rightarrow \K_n(R\text{-}\Mod)$ preserves products. Then Lemma \ref{compactly generated} and \ref{Bousfield} imply that $J$ has a left adjoint $J_\lambda$. Hence there are adjoint pairs
$$
\xymatrix{
  \K_n(R\text{-}\Mod ) \ar@<0.5ex>[r]^{J_\lambda}
     & \K_n(R\text{-}\Inj )\ar@<0.5ex>[l]^{J}\ar@<.5ex>[r]^Q&\D_n(R\text{-}\Mod ).\ar@<.5ex>[l]^{Q_\rho} }
$$
Since $J$ is a fully faithful right adjoint of $J_\lambda$, ~$\Hom_{\K_n(R\text{-}\Mod )}(\mathrm{Ker} J_\lambda,\K_n(R\text{-}\Inj ))=0$. This implies $\mathrm{Ker}J_\lambda\subseteq\K^{\ac}_n(R\text{-}\Mod)$. Thus for each $M\in \K_n(R\text{-}\Mod )$, the unit $\eta_M\colon M\rightarrow J_\lambda (M)$ is a quasi-isomorphism. It follows that $Q(M)\cong (Q\circ J_\lambda) (M)$. That is, $Q\circ J_\lambda$ is isomorphic to the localization functor $\K_n(R\text{-}\Mod)\rightarrow \D_n(R\text{-}\Mod ).$  By \ref{inclusion has right iff quo has}, $J\circ Q_\rho$ is fully faithful. As $J$ is also fully faithful, we infer that $Q_\rho$ is too. This completes the proof.
\end{proof}

\begin{chunk}
Let $\A$ be an abelian category. A $n$-periodic complex $X\in \K_n(\A)$ is called \emph{homotopy injective} (resp. \emph{homotopy projective}) if $$\Hom_{\K_n(\A)}(Y,X)=0~~ (\text{resp. } \Hom_{\K_n(\A)}(X,Y)=0)
$$ 
for each acyclic complex $Y\in \K_n(\A)$. Denote by $\K_n^{\mathsf{i}}(\A)$ (resp. $\K_n^{\mathsf{p}}(\A)$) the full subcategory of $\K_n(\A)$ consisting of all homotopy injective  (resp. homotopy projective) complexes. They naturally inherit the structure of triangulated categories.
\end{chunk}
Let $Q_{\rho}$ denote the right adjoint of $Q\colon \K_n(R\text{-}\Inj )\rightarrow \D_n(R\text{-}\Mod )$. Using the adjointness, it is easy to check  $Q_{\rho}(X)$ is homotopy injective for each $n$-periodic complex $X$ and the unit $X\rightarrow Q_{\rho}(X)$ is a quasi-isomorphism. Thus we get: 
\begin{corollary}\label{homotopy injective resolution}
$Q_{\rho}$ induces a triangle equivalence
$$
\D_n(R\text{-}\Mod )\stackrel{\sim}\longrightarrow\K_n^{\mathsf{i}}(R\text{-}\Mod).
$$
\end{corollary}

\begin{remark}\label{homotopy projective}
(1) Tang and Huang \cite[Theorem 5.11]{TH} proved an analog of the above result for higher differential objects. Two results coincide when $n=1$.

(2) Stai \cite[Section 3]{Stai} obtained the dual version of the above result. That is,  the localization functor $Q\colon \K_n(R\text{-}\Mod )\rightarrow \D_n(R\text{-}\Mod)$ has a left adjoint and the left adjoint induces a triangle equivalence 
$$
\D_n(R\text{-}\Mod)\stackrel{\sim}\longrightarrow \K_n^{\mathsf{p}}(R\text{-}\Mod).
$$
\end{remark}

\section{The triangulated hull of the orbit categories}\label{Section 3}
In this section, we prove Theorem \ref{t1} from the introduction.
\begin{chunk}
Let $\A$ be an additive category and $T\colon \A\rightarrow \A$ be an autoequivalence.
As mentioned in the introduction, the objects in the same $T$-orbit are isomorphic in the orbit category $\A/T$. We remind the reader that in general $F$ is not isomorphic to the identity functor in the orbit category $\A/T$;  see \cite{Keller09} and \cite[Proposition 5.6]{Stai}. However, there is a natural isomorphism $\pi\cong \pi\circ T$, where $\pi\colon \A\rightarrow \A/T$ is the projection functor. Moreover, this gives rise to the universal property of the orbit category: 

If the functor $F\colon \A\rightarrow\B $ satisfies $F\circ T\cong F$, then there exists a natural functor $\overline{F}\colon \A/T\rightarrow \B$ such that $\overline{F}\circ \pi= F$. 
\end{chunk}
\begin{chunk}
Let $\A$ be an additive category. Recall the degree shift functor (n) on $\C(\A)$: for a complex $X$, $X(n)^i:=X^{i+n}$, $\del^i_{X(n)}:=\del_X^{i+n}$; $(n)$ acts trivially on morphisms. There is a natural isomorphism  $X[n]\xrightarrow{\cong} X(n)$ which maps $x\in X^i$ to $(-1)^{ni}x$. 

If further $\A$ is an additive category with coproducts (resp. AB4 category), then $\Delta\circ [n]\cong \Delta\circ (n)=\Delta$. By the universal property of the orbit category, $\Delta$ induces 
$$
\overline{\Delta}\colon \K(\mathcal{A})^{\com}/[n]\longrightarrow \K_n(\mathcal A)^{\com}~~(\text{resp. }\overline{\Delta}\colon \D(\mathcal{A})^{\com}/[n]\longrightarrow \D_n(\mathcal A)^{\com}).
$$
\end{chunk}
We first strength Lemma \ref{compactly generated} to the following result; compare \cite[Lemma 3.13]{Stai}.
\begin{proposition}\label{embedding}
(1) Let $\mathcal A$ be an additive category with coproducts. If $\K(\A)$ is compactly generated, then there is a fully faithful embedding
$$
\overline{\Delta}\colon \K(\mathcal{A})^{\com}/[n]\longrightarrow \K_n(\mathcal A)^{\com} 
$$
and $\K_n(\mathcal{A})$  is compactly generated by its image.

(2) Let $\mathcal A$ be an AB4 category. If $\D(\A)$ is compactly generated, then there is a fully faithful embedding
$$
\overline{\Delta}\colon \D(\mathcal{A})^{\com}/[n]\longrightarrow \D_n(\mathcal A)^{\com} 
$$
and $\D_n(\mathcal{A})$  is compactly generated by its image.
\end{proposition}
\begin{proof}
We prove (1). The proof of (2) is similar. By Lemma \ref{compactly generated}, it remains to show $\overline{\Delta}$ is fully faithful. For $X,Y\in \K(\mathcal A)^{\com}$, we have
\begin{align*}
    \Hom_{\K_n(\A)}(\Delta(X),\Delta(Y))\cong &\Hom_{\K(\mathcal A)}(X,\nabla \Delta(Y))\\
    \cong &\coprod_{i\in \mathbb{Z}}\Hom_{\K(\mathcal A)}(X,Y[ni]),
\end{align*}
where the second isomorphism is because $X$ is compact and $\nabla\Delta (Y)\cong \coprod_{i\in \mathbb Z}Y[ni]$ (see \ref{counit}). It follows immediately from the isomorphism above that the induced functor $\overline{\Delta}\colon  \K(\mathcal{A})^{\com}/[n]\rightarrow \K_n(\mathcal A)^{\com}
$
is fully faithful.
\end{proof}
\begin{chunk}
A homogeneous morphism $f$ in a dg category is called \emph{closed} if $f$ is of degree $0$ and $\del(f)=0$. We call a natural transformation $\eta$ between dg functors \emph{closed} if $\eta_X$ is closed for all $X\in \A$. 
\end{chunk}
The following is the universal property of the orbit category of the dg category.
\begin{lemma}\label{induce a functor}
Let $T\colon \A\rightarrow \A$ be a dg autoequivalence of a dg category $\A$ and $F\colon \A\rightarrow \B$ be a dg functor between dg categories such that there exists a closed natural isomorphism $\eta\colon F\circ T\rightarrow F$. Then $F$ induces a dg functor $\overline{F}\colon \A/T\rightarrow \B$ such that $\overline{F}\circ \pi=F$, where $\pi\colon \A\rightarrow \A/T$ is the projection functor.
\end{lemma}
\begin{proof}
By assumption, $\eta$ induces closed isomorphisms $F\circ T^i\stackrel{\sim}\rightarrow  F$ (denoted $\eta^i$). We define $\overline{F}\colon \A/T\rightarrow \B$ as follows: $\overline{F}(M)=F(M)$ for all $M\in \A$; for each homogeneous morphism $\alpha\colon M\rightarrow T^j(N)$ in $\Hom_{\A/T}(M,N)$, $\overline{F}(\alpha)$ is defined by the following composition
$$
\xymatrix{F(M)\ar[r]^-{F(\alpha)}& F(T^j(N))\ar[r]^-{\eta^j_N}_-{\cong}& F(N).
}
$$
By assumption $F$ is dg functor and $\eta$ is closed, we have a commutative diagram
$$
\xymatrix{
\Hom_\A(X,T^j(Y))\ar[r]^-F\ar[d]^-{\del}& \Hom_\B(F(X),F(T^j(Y)))\ar[r]^-{(\eta^j_Y)_\ast}\ar[d]^-{\del}& \Hom_\B(F(X), F(Y))\ar[d]^-{\del}\\
\Hom_\A(X,T^j(Y))\ar[r]^-{F}& \Hom_\B(F(X),F(T^j(Y)))\ar[r]^-{(\eta^j_Y)_{\ast}}&\Hom_\B(F(X),F(Y)).
}
$$
This means $\overline{F}$ is a dg functor. Clearly $\overline{F}\circ \pi=F$.
\end{proof}

\begin{example}\label{example}
Let $\mathcal{C}$ be an additive category. Set $\A=\C_{\dg}(\mathcal{C})$ and $\B=\C_{\dg}(\Z)$. The suspension functor $[n]\colon \C_{\dg}(\mathcal{C})\rightarrow \C_{\dg}(\mathcal{C})$ is a dg autoequivalence. If $X$ is a $n$-periodic complex in $\A$, 
 then
$$
 \Hom_\A(Y[n],X)\cong \Hom_\A(Y(n),X)\cong\Hom_\A(Y,X),
$$
where the first isomorphism is induced by  $Y[n]\cong Y(n)$ and the second one maps $\alpha\colon Y(n)^i\rightarrow X^j$ to $\alpha\colon Y^{n+i}\rightarrow X^{n+j}$.
Set $F=\Hom_\A(-,X)$ and $T=[n]$. We conclude that there is a closed natural isomorphism $F\circ T\cong F$. This is an example that satisfies the assumption of Lemma \ref{induce a functor}.
\end{example}
\begin{chunk}\label{compact object in K(Inj R)}
Let $R$ be a left noetherian ring.
Krause \cite[Proposition 2.3]{Krause} proved that $\K(R\text{-}\Inj)$ is a compactly generated. Moreover, he observed that the localization functor $\K(R\text{-}\Mod )\rightarrow \D(R\text{-}\Mod )$ induces a triangle equivalence
$$
    \K(R\text{-}\Inj )^{\com}\stackrel{\sim}\longrightarrow \D^b(R\text{-}\mo).
$$
The inverse is induced by taking injective resolution. In particular, $\K(R\text{-}\Inj )^{\com}$ is the full subcategory of $\K(R\text{-}\Inj )$ consisting of complexes with finitely generated total cohomology. 
\end{chunk}
 Recall that $\per_{\dg}(R)$ is the dg category of perfect complexes over $R$ and  $\C_{\dg}^{+,f}(R\text{-}\Inj)$ is the dg category of bounded below complexes of injective $R$-modules with finitely generated total cohomology. They are dg enhancements of $\per(R)$ and $\D^b(R\text{-}\mo )$ respectively; see \ref{compact object in K(Inj R)} and Example \ref{example of dg enhancement}. 
 
Next, we realize examples of triangulated categories in Theorem \ref{recollement} as derived categories of dg categories; compare \cite[Theorem 2.2]{CLW} and \cite[Appendix A]{Krause}.

\begin{theorem}\label{main result}
Let $R$ be a left noetherian ring. There are triangle equivalences
$$
\K_n(R\text{-}\Inj )\stackrel{\sim}\longrightarrow \D(\C_{\dg}^{+,f}(R\text{-}\Inj )/[n])
\text{ and }~\D_n(R\text{-}\Mod )\stackrel{\sim}\longrightarrow\D(\per_{\dg} (R)/[n]).$$
\end{theorem}
\begin{proof}
We prove the first equivalence. The proof of the second one is similar. 
For each complex $X$ of $R$-modules, set $X^{\wedge}=\Hom_R(-,X)$. By Lemma \ref{induce a functor}  and Example \ref{example}, the map $I\mapsto \overline{{I^{\wedge}}}\mid_{\C_{\dg}^{+,f}(R\text{-}\Inj )/[n]}$ induces an exact functor 
$$\Phi\colon \K_n(R\text{-}\Inj)\longrightarrow \D(\C_{\dg}^{+,f}(R\text{-}\Inj)/[n]).$$
The functor $\Phi$ preserves coproducts. Indeed, for each object $J\in \C_{\dg}^{+,f}(R\text{-}\Inj)$ and a family $I_i\in \K_n(R\text{-}\Inj)$ ($i\in S$), we have isomorphisms
\begin{align*}
    \h^l(\Phi(\coprod_{i\in S}I_i)(J))\cong & \Hom_{\K(R\text{-}\Inj)}(J[-l], \coprod_{i\in S}I_i)\\
    \cong& \coprod_{i\in S}\Hom_{\K(R\text{-}\Inj)}(J[-l],I_i)\\
    \cong&\coprod_{i\in S}\h^l(\Phi(I_i)(J))\\
    \cong &\h^l(\coprod_{i\in S}\Phi(I_i)(J))
\end{align*}
for each $l\in \Z$,
where the second isomorphism is because $J$ is a compact object in $\K(R\text{-}\Inj)$; see \ref{compact object in K(Inj R)}. Hence in $\D(\C_{\dg}^{+,f}(R\text{-}\Inj )/[n])$,
$$
\Phi(\coprod_{i\in S}I_i)\cong \coprod_{i\in S}\Phi(I_i).
$$
We observe that there exists a commutative diagram
\begin{equation}\label{comm}
\xymatrix{
\h^0(\C_{\dg}^{+,f}(R\text{-}\Inj/[n])\ar[r]^-{\overline{\Delta}}\ar[d]_-{\Y}& \K_n(R\text{-}\Inj )^{\com}\ar[d]^-{\mathrm{inc}}\\
\D(\C_{\dg}^{+,f}(R\text{-}\Inj)/[n])& \K_n(R\text{-}\Inj)\ar[l]_-{\Phi}
},
\end{equation}
where $\Y$ is the Yoneda embedding. 
From Proposition \ref{embedding},  $$\overline{\Delta} \colon \h^0(\C_{\dg}^{+,f}(R\text{-}\Inj)/[n])\longrightarrow \K_n(R\text{-}\Inj)^{\com}$$ 
is fully faithful and $\K_n(R\text{-}\Inj)$ is compactly generated by the image of $\overline{\Delta}$.   As $\D(\C_{\dg}^{+,f}(R\text{-}\Inj)/[n])$ is compactly generated by the image of $\Y$, we conclude that $\Phi$ is an equivalence by Lemma \ref{test equivalence}.
\end{proof}
\begin{remark}
(1) Let $k$ be a field. When $R$ is a finite dimensional $k$-algebra with finite global dimension, the triangle equivalence $\D_n(R\text{-}\Mod)\stackrel{\sim}\rightarrow\D(\per_{\dg} (R)/[n])$ was proved by Stai with a different method; see \cite[Section 4]{Stai}. 

(2) Let $\B$  $(\text{resp. } \A)$ denote $\per_{\dg}(R)/[n]$ $(\text{resp. } \C_{\dg}^{+,f}(R\text{-}\Inj)/[n])$. We can regard $\B$ as a full dg subcategory of $\A$; see \ref{compact object in K(Inj R)}. Then we can form a dg quotient category $\A/\B$; see Keller's construction in \cite[Section 4]{Keller99}. 
The restriction functor $\D(\A/\B)\rightarrow \D(\A)$ is fully faithful and it's essential image is equal to the kernel of the restriction functor $
\D(\A)\rightarrow \D(\B)$; see \cite[Section 4]{Keller99} and \cite[Proposition 4.6]{Drinfeld}. Combine this with Theorem \ref{recollement} and Theorem \ref{main result}, we conclude that there is a triangle equivalence 
$$
\K_n^{\ac}(R\text{-}\Inj)\simeq \D(\A/\B).
$$
\end{remark}
\begin{chunk}\label{hull}
Let  $\A$ be a dg enhancement of a triangulated category $\T$. Assume the functor $F\colon \T\rightarrow \T$ is an autoequivalence and it lifts to a dg equivalence $ \A\rightarrow \A$ (still denoted $F$). Then we can form an orbit category $\A/F$ which naturally inherits a structure of the dg category and gives the desired enhancement of $\T/F$. Hence
$$
\T/F\stackrel{\sim}\longrightarrow \h^0(\A/F)\stackrel{\Y}\longrightarrow \D(\A/F),
$$
where $\Y$ is the Yoneda embedding. The \emph{triangulated hull} of $\T/F$ is chosen to be the triangulated subcategory of $\D(\A/F)$ generated by the image of $\Y$. It is up to direct summands equivalent to $\D(\A/F)^{\com}$. Thus we use $\D(\A/F)^{\com}$ to represent the triangulated hull of $\T/F$ in the article; see  Keller's definition in \cite[Section 5]{Keller05} for a broader definition of the triangulated hull.
\end{chunk}

The inverse of the equivalence $\K(R\text{-}\Inj)^{\com}\stackrel{\sim}\rightarrow \D^b(R\text{-}\mo)$ (see \ref{compact object in K(Inj R)}) is induced by taking injective resolution. We denote it by $\mathsf{i}$.

\begin{corollary}\label{embedding in hull}
Let $R$ be a left noetherian ring. Induced by the  compression of complexes, the functors 
$$
\overline{\Delta\circ \mathsf{i}}\colon \D^b(R\text{-}\mo)/[n]\longrightarrow \K_n(R\text{-}\Inj )^{\com}\text{ and }~\overline{\Delta}\colon\per(R)/[n]\longrightarrow \D_n(R\text{-}\Mod )^{\com}
$$
are the embedding of the orbit categories into their triangulated hull. 
\end{corollary}
\begin{proof}
For the first one: $\C^{+,f}_{\dg}(R\text{-}\Inj )$ is the dg enhancement of $\D^b(R\text{-}\mo)$. Then  $\C_{\dg}^{+,f}(R\text{-}\Inj)/[n]$ is the desired dg enhancement of $\D^b(R\text{-}\mo)/[n]$.  Given this and \ref{hull}, the desired result follows from Theorem \ref{main result}.

For the second one: $\per_{\dg}(R)$ is the dg enhancement of $\per (R)$. Then the remaining proof is parallel to the first one.
\end{proof}

\begin{remark}\label{Grothendieck category}
Fix a \emph{locally noetherian} Grothendieck category $\A$. That is, $\A$ is an AB4 category with exact direct colimit, and $\A$ has a set $\A_0$ of noetherian objects such that every object in $\A$ is a quotient of a coproduct of objects in $\A_0$. Denote by $\A\text{-}\text{noeth}$ the full subcategory of $\A$ formed by noetherian objects, and by $\A\text{-}\Inj$ the full subcategory of $\A$ formed by injective objects. With the same argument of Theorem \ref{main result}, we have 
$
\K_n(\A\text{-}\Inj)\stackrel{\sim}\rightarrow \D(\C_{\dg}^{+,f}(\A\text{-}\Inj)/[n]),
$
where $\C_{\dg}^{+,f}(\A\text{-}\Inj)$ is the dg category of bounded below complexes of injective objects with noetherian total cohomology. Then the same proof of Corollary \ref{embedding in hull} yields that the compression of complexes 
$$
\overline{\Delta\circ \mathsf{i}}\colon \D^b(\A\text{-noeth})/[n]\longrightarrow \K_n(\A\text{-}\Inj )^{\com}
$$
induces an embedding of the orbit category into its triangulated hull.
\end{remark}

\begin{chunk}\label{finite global dimension case}
It was proved by Stai \cite[Lemma 3.5]{Stai} when $R$ has finite global dimension, then every object in $\C_1(R\text{-}\mo)$ is quasi-isomorphic to one admitting a finite projective flag (see definition in Example \ref{finite projective flag}). Thus $\D_1(R\text{-}\mo)$ is equal to the thick subcategory generated by $\Delta(R)$. As he also mentioned, this extends to any $n\geq 1$. Combine with Proposition \ref{embedding}, there is a natural triangle equivalence
$$
\D_n(R\text{-}\mo)\stackrel{\sim}\longrightarrow\D_n(R\text{-}\Mod)^{\com}.
$$

With the same method of Stai, one can show: if $\A$ is an abelian category with enough projective objects and every object in $\A$ has finite projective dimension, then
$$
\D_n(\A)=\thick_{\D_n(\A)}(\{\Delta(P)\mid P \text{ is projective in } \A\}).
$$
\end{chunk}

The following result was proved independently by Stai \cite[Theorem 4.3]{Stai} and Zhao \cite[Theorem 2.10]{Zhao} when $R$ is a finite  dimensional
algebra with finite global dimension over a field.
\begin{corollary}\label{SZ}
Let $R$ be a left noetherian ring with finite global dimension, then the compression of complexes
$$
\overline{\Delta}\colon\D^b(R\text{-}\mo)/[n]\longrightarrow \D_n(R\text{-}\mo)
$$
is an embedding of the orbit category into its triangulated hull.
\end{corollary}
\begin{proof}
When $R$ has finite global dimension, $\D^b(R\text{-}\mo)=\per (R)$. Combine with \ref{finite global dimension case}, the desired result follows from Corollary \ref{embedding in hull}.
\end{proof}
\begin{chunk}
When $R$ is hereditary, $\D^b(R\text{-}\mo)/[n]$ is triangulated and hence it is (up to direct summands) equivalent to its triangulated hull; see \cite[Theorem 1]{Keller05} and \cite[Proposition 5.3]{Stai}. When $n=1$ and $R$ is a path algebra of finite connected acyclic quiver, Ringel and Zhang \cite[Theorem 1]{RZ} proved that $\D^b(R\text{-}\mo)/[1]$ is equivalent to a stable category of certain Frobenius category.
\end{chunk}

\section{Derived equivalence as derived tensor product}\label{Section 4}

For two rings $A$ and $B$, the purpose of this section is to compare the triangle equivalences $\D(A\text{-}\Mod)\simeq  \D(B\text{-}\Mod)$  and $\D_n(A\text{-}\Mod)\simeq \D_n(B\text{-}\Mod)$. It turns out that these two equivalences are closely related; see Proposition \ref{derived equivalence}.

\begin{chunk}\label{definition of tensor product}
\textbf{Tensor products.}
Let $X$ be a complex of $B\text{-}A$ bimodules. For a $n$-periodic complex $Y$ in $\C_n(A\text{-}\Mod)$, the tensor product $X\otimes_A Y$ is a $n$-periodic complex in $\C_n(B\text{-}\Mod)$.
Thus $X\otimes_A-$ gives a functor
$$
X\boxtimes_A-\colon\C_n(A\text{-}\Mod)\longrightarrow \C_n(B\text{-}\Mod).
$$
The notation $X\boxtimes_A-$ is to distinguish it from $X\otimes_A-\colon \C(A\text{-}\Mod)\rightarrow \C(B\text{-}\Mod)$.
Moreover, the following diagram 
\begin{equation}\label{tensor commutes with Delta}
    \xymatrix{
\C(A\text{-}\Mod)\ar[d]_-{\Delta}\ar[r]^-{X\otimes_A-}&\C(B\text{-}\Mod)\ar[d]^-{\Delta}\\
\C_n(A\text{-}\Mod)\ar[r]^-{X\boxtimes_A-}& \C_n(B\text{-}\Mod).
}
\end{equation}
is commutative; see \cite[(1.9.4)]{ABI} for the case $n=1$.
\end{chunk}

\begin{chunk}\label{derived tensor}
Keep the same assumption as \ref{definition of tensor product}. Since $X\boxtimes_A-$ preserves homotopy, suspensions and mapping cones, it induces an exact functor $X\boxtimes_A-\colon \K_n(A\text{-}\Mod)\rightarrow \K_n(B\text{-}\Mod)$.
We define the \emph{derived tensor product} $X\boxtimes_A^{\mathsf{L}}-$ by the following composition
$$
\xymatrix{
\D_n(A\text{-}\Mod)\ar[r]^-{\mathsf{p}}& \K_n(A\text{-}\Mod)\ar[r]^-{X\boxtimes_A-}& \K_n(B\text{-}\Mod)\ar[r]^-{Q}&\D_n(B\text{-}\Mod)
}
$$
where $\mathsf{p}$ is the left adjoint of the canonical functor $\K_n(A\text{-}\Mod)\rightarrow \D_n(A\text{-}\Mod)$; see Remark \ref{homotopy projective} for its existence.
The compression functor $\Delta\colon \K(A\text{-}\Mod)\rightarrow 
\K_n(A\text{-}\Mod)$ preserves homotopy projective objects because its right adjoint preserves acyclic complexes. Combine this with (\ref{tensor commutes with Delta}), we observe that there exists a commutative diagram
$$
\xymatrix{
\D(A\text{-}\Mod)\ar[d]_-{\Delta}\ar[r]^-{X\otimes^{\mathsf{L}}_A-}&\D(B\text{-}\Mod)\ar[d]^-{\Delta}\\
\D_n(A\text{-}\Mod)\ar[r]^-{X\boxtimes_A^{\mathsf{L}}-}& \D_n(B\text{-}\Mod).
}
$$
\end{chunk}
For a triangulated category $\T$, we write $\Sigma_{\T}$ to be the suspension functor of $\T$.
\begin{lemma}\label{full faithful property}
Let $F\colon \T\rightarrow \T^\prime$ be an exact functor between triangulated categories. Then $F$ is fully faithful if and only if the induced functor $\overline{F}\colon\T/\Sigma_{\T}^n\rightarrow \T^\prime/\Sigma_{\T^\prime}^n$ is fully faithful. Moreover, $F$ is an equivalence if and only if $\overline{F}$ is an equivalence.
\end{lemma}
\begin{proof}
Since $\overline{F}(X)=F(X)$ for each object $X\in \T$, the second statement follows from the first one. Fix objects $X,Y\in \T$, we observe that the map 
$$
\overline{F}\colon \coprod_{i\in \Z}\Hom_{\T}(X,\Sigma_{\T}^{ni}Y)\longrightarrow \coprod_{i\in \Z}\Hom_{\T^\prime}(F(X),\Sigma_{\T^\prime}^{ni}F(Y))
$$
is the direct sum of the following composition maps
$$
\Hom_{\T}(X,\Sigma_{\T}^{ni}Y)\stackrel{F}\longrightarrow\Hom_{\T^\prime}(F(X),F(\Sigma_{\T}^{ni}Y))\cong \Hom_{\T^\prime}(F(X),\Sigma_{\T^\prime}^{ni}F(Y)),
$$
where the isomorphisms are induced by the canonical isomorphism $F\Sigma_{\T}\cong \Sigma_{\T^\prime}F$. The desired result follows.
\end{proof}

\begin{lemma}\label{induce equivalence}
Let $F,G,\Phi_1,\Phi_2$ are exact functors between compactly generated triangulated categories such that the following diagram
$$
\xymatrix{
\S_1\ar[r]^-F\ar[d]_-{\Phi_1}& \S_2\ar[d]^-{\Phi_2}\\
\T_1\ar[r]^-G&\T_2
}
$$
commutes.
Assume $F,G$ preserves coproducts and $\Phi_i$ preserves compact objects for $i=1,2$. Moreover, we assume $\Phi_i$ induces a fully faithful functor
$$
\overline{\Phi_i}\colon \S_i^{\com}/\Sigma_{\S_i}^n\longrightarrow \T_i^{\com}
$$
such that $\T_i$ is compactly generated by its image.
Then we have implications:

(1) $F$ is an equivalence $\Rightarrow G$ is an equivalence.

(2) $F$ preserves compact objects and $G$ is an equivalence $\Rightarrow F$ is fully faithful.
\end{lemma}
\begin{proof}
Combine with the assumption, the condition of (1) or (2) implies the diagram
$$
    \xymatrix{
\S_1^{\com}/\Sigma_{\S_1}^n\ar[r]^-{\overline{F}}\ar[d]_-{\overline{\Phi_1}}& \S_2^{\com}/\Sigma_{\S_2}^n\ar[d]^-{\overline{\Phi_2}}\\
\T_1^{\com}\ar[r]^-G&\T_2^{\com}.
}
$$
commutes. Indeed, this is trivial for (2). For (1), it remains to show that $G$ preserves compact objects. The assumption and the condition of (1) yield $G(\Ima\overline{\Phi_1})\subseteq \T_2^{\com}$. Since $\T_1$ is compactly generated by $\Ima \overline{\Phi_1}$, we have $\thick_{\T_1}(\Ima \overline{\Phi_1})=\T_1^{\com}$; see \ref{Neeman's description of compact objects}. On the other hand,  the full subcategory
$
\{X\in \T_1\mid G(X)\in \T_2^{\com}\}
$
of $\T_1$ is thick. Thus $G$ preserves compact objects.

(1) Assume $F$ is equivalence. Then $\overline{F}$ is an equivalence. For $i=1,2$,  $\Ima \overline{\Phi_i}$ is a compact generating set of $\T_i$. Clearly $\Ima \overline{\Phi_i}$ is closed under suspensions. Then we apply Lemma \ref{test equivalence} to conclude that $G\colon \T_1\rightarrow \T_2$ is an equivalence.

(2) By assumption, $G$ induces an equivalence $G\colon \T_1^{\com}\stackrel{\sim}\longrightarrow\T_2^{\com}.$ 
Combine with $\overline{\Phi_i}$ is fully faithful for $i=1,2$, we get that $\overline{F}$ is fully faithful. Then Lemma \ref{full faithful property} yields the functor $F\colon \S_1^{\com}\rightarrow \S_2^{\com}$ is fully faithful. According to Lemma \ref{test equivalence}, $F$ is fully faithful.
\end{proof}
\begin{example}
Let $R$ be a commutative noetherian ring with a dualizing complex $\omega$. Iyengar and Krause \cite[Theorem I]{IY} proved that  $$\omega\otimes_R-\colon \K(R\text{-}\Proj)\longrightarrow \K(R\text{-}\Inj)$$ is a triangle equivalence.  Combine this result with Proposition \ref{embedding} and Lemma \ref{induce equivalence}, we immediately get that there is a triangle equivalence $$\omega\boxtimes_R-\colon \K_n(R\text{-}\Proj)\stackrel{\sim}\longrightarrow \K_n(R\text{-}\Inj).$$

\end{example}

Let $\Thick \T$ be the \emph{lattice of thick subcategories} of a triangulated category $\T$.
\begin{chunk}\label{maps of lattices}
Suppose $\A$ is an additive category with coproducts (\text{resp.} AB4 category). We write $\T$ to be $\K(\A)$ (\text{resp.} $\D(\A)$) and $\T^{\prime}$ to be $\K_n(\A)$ (\text{resp.} $\D_n(\A)$). For a thick subcategory $\S$ of $\T^{\com}$, we let $F(\S)$ be the smallest thick subcategory of $\T^{\prime\com}$ containing all objects $\Delta(X)$ such that $X\in\S$. For a thick subcategory $\S^\prime$ of $\T^{\prime\com}$, we let $G(\S^{\prime})$ be the smallest thick subcategory of $\T^{\com}$ containing all objects $X$ in $\T^{\com}$ such that $\Delta(X)\in \S^{\prime}.$ Thus we have maps of lattices
\[\begin{tikzcd}
\Thick \T^c\arrow[r,shift left=0.8ex,"F"]&\Thick T^{\prime\com}\arrow[l,shift left=0.8ex,"G"]
\end{tikzcd}.\]
\end{chunk}
Next result is inspired by a recent result of Iyengar, Letz, Pollitz and the author \cite[Corollary 5.9]{ILLP}. It is important in the proof of Proposition \ref{derived equivalence}.
\begin{lemma}\label{injection of map of lattices}
Keep the assumptions as \ref{maps of lattices}.  
 Then $G\circ F=\id$. In particular, the map of lattices $F\colon \Thick \T^{\com} \rightarrow \Thick \T^{\prime\com}$ is injective.
\end{lemma}
\begin{proof}
Fix a thick subcategory $\S$ of $\T^c$. In order to show $GF(\S)=\S$, it suffices to show for $X,Y\in \T^c$,
$$
X\in \thick_{\T}(Y)\iff \Delta(X)\in \thick_{\T^\prime}(\Delta(Y)).
$$
The forward direction is trivial; see \cite[Lemma 2.4]{ABIM}. For the converse, assume $\Delta(X)$ is an object in $\thick_{\T^\prime}(\Delta(Y))$. Then we have $\nabla\Delta(X)\in \thick_{\T}(\nabla\Delta(Y))$. Since $\nabla\Delta(M)\cong \coprod_{i\in \Z}M[ni]$ for each $M\in \T$ (see \ref{counit}), $X$ is in the localizing subcategory of $\T$ generated by $Y$. As $X, Y$ are compact objects in $\T$, we conclude by \ref{Neeman's description of compact objects} that $X$ is in $\thick_{\T}(Y)$. As required.
\end{proof}

\begin{proposition}\label{derived equivalence}
Let $A,B$ be two rings and $X$ be a complex of $B\text{-}A$-bimodules. Then the functor $X\otimes_A^{\mathsf{L}}-\colon\D(A\text{-}\Mod )\rightarrow \D(B\text{-}\Mod)$ is a triangle equivalence if and only if the functor $X\boxtimes_A^{\mathsf{L}}-\colon\D_n(A\text{-}\Mod)\rightarrow\D_n(B\text{-}\Mod)$ is a triangle equivalence. 
\end{proposition}
\begin{proof}
First, assume $X\otimes_A^{\mathsf{L}}-$ is a triangle equivalence.  It follows immediately from Proposition \ref{embedding}, \ref{derived tensor} and Lemma \ref{induce equivalence} that $X\boxtimes_R^{\mathsf{L}}-$ is a triangle equivalence.

Now, assume $X\boxtimes_A^{\mathsf{L}}-$ is a triangle equivalence. It restricts to an equivalence between the full categories of compact objects. Combine with the commutative diagram in \ref{derived tensor}, we conclude by Lemma \ref{description of compact object} that $X\otimes_A^{\mathsf{L}}-\colon \D(A\text{-}\Mod)\rightarrow \D(B\text{-}\Mod)$ preserves compact objects. It follows from Proposition \ref{embedding} and Lemma \ref{induce equivalence} that $X\otimes_A^{\mathsf{L}}-$ is fully faithful. 

To show $X\otimes_A^{\mathsf{L}}-$ is equivalence, by Lemma \ref{test equivalence} it remains to show the essential image of $X\otimes_A^{\mathsf{L}}-\colon \D(A\text{-}\Mod)^{\com}\rightarrow \D(B\text{-}\Mod)^{\com}$, denoted $\S$, is a compact generating set of $\D(B\text{-}\Mod)$. Consider the commutative diagram
$$
\xymatrix{
\D(A\text{-}\Mod)^{\com}\ar[d]_-{\Delta} \ar[r]^-{X\otimes_A^{\mathsf{L}}-}& \D(B\text{-}\Mod)^{\com}\ar[d]^-{\Delta}\\
\D_n(A\text{-}\Mod)^{\com}\ar[r]_{\sim}^-{X\boxtimes_A^{\mathsf{L}}-}& \D_n(B\text{-}\Mod)^{\com}\\
}
$$
we apply Proposition \ref{embedding} to get $\thick_{\D_n(B\text{-}\Mod)}(\Delta(\S))=\D_n(B\text{-}\Mod)^{\com}$. Then Lemma \ref{injection of map of lattices} yields the smallest thick subcategory of $\D(B\text{-}\Mod)^{\com}$ containing $\S$ is the whole of $\D(B\text{-}\Mod)^{\com}$. Hence $\S$ is a compact generating set of $\D(B\text{-}\Mod)$. As required.
\end{proof}

Two rings are \emph{derived equivalent} provided that $\D(A\text{-}\Mod)$ and $\D(B\text{-}\Mod)$ are equivalent as triangulated categories.
\begin{chunk}\label{standard equivalence}
It is an open question that whether any triangle equivalence
$$
\D(A\text{-}\Mod)\stackrel{\sim}\longrightarrow\D(B\text{-}\Mod)
$$ 
is isomorphic to a derived tensor functor $X\otimes_A^{\mathsf{L}}-$, where $X$ is a complex of $B\text{-}A$ bimodules. Such derived equivalence is called a \emph{standard equivalence}. 

However, if $A,B$ are two algebras over a commutative ring $k$ such that they are flat as $k$-modules, then any triangle equivalence 
$\D(A\text{-}\Mod)\stackrel{\sim}\rightarrow\D(B\text{-}\Mod)$ 
is standard; see \cite[Corollary 9.2]{Keller94} and \cite[Section 3]{Rickard1991}.
\end{chunk}
Combine with Proposition \ref{derived equivalence},  the statement in \ref{standard equivalence} implies the following result.
\begin{corollary}\label{implication}
Let $k$ be a commutative ring and $A,B$ be flat $k$-algebras. If $A$ and $B$ are derived equivalent, then $\D_n(A\text{-}\Mod)$ and $\D_n(B\text{-}\Mod)$ are equivalent as triangulated categories.
\end{corollary}

If $A$ is a left noetherian ring with finite global dimension, then $\D_n(A\text{-}\Mod)^{\com}=\D_n(A\text{-}\mo)$; see \ref{finite global dimension case}. As a consequence of Corollary \ref{implication}, we have:
\begin{corollary}\label{finite global}
Let $k$ be a commutative ring and $A,B$ be flat $k$-algebras. If $A,B$ are noetherian with finite global dimensions and $A,B$ are derived equivalent, then $\D_n(A\text{-}\mo)$ and $\D_n(B\text{-}\mo)$ are equivalent as triangulated categories.
\end{corollary}
\begin{remark}
The above corollary extends a result of Zhao \cite[Theorem]{Zhao}. In her paper, she proved the above result holds for finite dimensional algebras with finite global dimensions over a field.
\end{remark}

\section{Koszul duality for periodic complexes}\label{Section 5}
Throughout this section, $k$ is a field and $S$ is the graded polynomial algebra $k[x_1,\ldots,x_c]$ with $\mathrm{deg}(x_i)=1$. We let $\Lambda$ denote the Koszul dual of $S$. More precisely, $\Lambda$ is the graded exterior algebra over $k$ on variables $\xi_1,\ldots,\xi_c$ of degree $-1$.

For a graded algebra $A$, denote by $A\text{-}\Gr$ (resp. $A\text{-}\gr$) the category of left (resp. finitely generated left) graded $A$-modules. A graded $A$-module is called \emph{graded-injective} provided that it is an injective object in $A\text{-}\Gr$. It is well-known that $A\text{-}\Gr$ has enough projective objects and enough injective objects; see \cite[Section 1.5 and Theorem 3.6.2]{BH}. 

Let $A\text{-}\GrInj$ denote the category of graded-injective $A$-module. As $\Lambda$ is noetherian, one can show that the direct sum of graded-injective $\Lambda$-module is graded-injective; the proof is parallel to the non-graded version \cite[Theorem 3.1.17]{EJ}.

The main purpose of this section is to give the following periodic version of the Koszul duality.
\begin{theorem}\label{Koszul duality}
There exists a triangle equivalence 
$$ 
\K_n(\Lambda\text{-}\GrInj )\stackrel{\sim}\longrightarrow \D_n(S\text{-} \Gr ).
$$
\end{theorem}
We give the proof of the above result at the end of this section. As a consequence, we have:

\begin{corollary}\label{version of BGG}
There is an embedding
$$
\D^b(\Lambda\text{-}\gr)/[n]\longrightarrow \D_n(S\text{-}\gr)
$$ 
of the orbit category into its triangulated hull.
\end{corollary}

Before giving the proof of the corollary, we recall a result.
\begin{chunk}\label{compact in K(GrInj)}
Due to Krause \cite[Proposition 2.3]{Krause}, $\K(\Lambda\text{-}\GrInj)$ is compactly generated. Moreover, the localization functor $\K(\Lambda\text{-}\Gr)\rightarrow \D(\Lambda\text{-}\Gr)$ induces a triangle equivalence
$$
\K(\Lambda\text{-}\GrInj)^{\com}\stackrel{\sim}\longrightarrow \D^b(\Lambda\text{-}\gr).
$$
Its inverse is induced by taking grade-injective resolution, denoted $\mathsf{i}$.
\end{chunk}
\begin{proof}[Proof of Corollary \ref{version of BGG}]
Keep the notation as \ref{compact in K(GrInj)}, Remark \ref{Grothendieck category} implies that the compression 
$$
\overline{\Delta\circ \mathsf{i}}\colon \D^b(\Lambda\text{-}\gr)/[n]\longrightarrow \K_n(\Lambda\text{-}\GrInj)^{\com}
$$
induces an embedding of $\D^b(\Lambda\text{-}\gr)/[n]$ into its triangulated hull. It follows from Theorem \ref{Koszul duality} that $\K_n(\Lambda\text{-}\GrInj)^{\com}$ is triangle equivalent to $\D_n(S\text{-}\Gr)^{\com}$. Choose $\A=S\text{-}\gr$ in \ref{finite global dimension case}, we conclude that 
$\D_n(S\text{-}\gr)$ is the smallest thick subcategory containing $\Delta(S(i))$ for all $i\in \Z$. It is precisely $\D_n(S\text{-}\Gr)^{\com}$; see \ref{Neeman's description of compact objects} and Proposition \ref{embedding}. This completes the proof.
\end{proof}
\begin{chunk}
Recall the functor $\Phi\colon\C(S\text{-}\Gr)\rightarrow\C(\Lambda\text{-}\Gr)$; see \cite{BGS} or \cite{EFS} for more details. Set $(-)^\ast:=\Hom_k(-,k)$. For a graded $S$-module $M=\coprod_{i\in \Z}M_i$, $\Phi(M)$ is defined by the complex
$$
\cdots\xrightarrow{\del}\Lambda^\ast\otimes_kM_{i-1}\xrightarrow {\del} \Lambda^\ast\otimes_k M_i\xrightarrow{\del}\Lambda^\ast\otimes_k M_{i+1}\xrightarrow{\del}\cdots,
$$
where $\del(f\otimes m):=(-1)^{l+i}\sum_{j=1}^c\xi_jf\otimes x_jm$ for $f\in (\Lambda^\ast)_l$ and $m\in M_i$; the sign makes sure that $\del$ is $\Lambda$-linear.
For a complex $M:\cdots\xrightarrow{d}M^{j-1}\xrightarrow dM^j\xrightarrow dM^{j+1}\xrightarrow d\cdots$ in $\C(S\text{-}\Gr)$, $\Phi(M)$ is defined by the total complex of the double complex
\begin{equation}\label{total complex}
    \xymatrix{
&\vdots\ar[d]^-{1\otimes d}&\vdots\ar[d]^-{1\otimes d}&\\
\cdots\ar[r]^-{\del}& \Lambda^\ast\otimes_k M^j_i\ar[r]^-{\del}\ar[d]^-{1\otimes d}& \Lambda^\ast\otimes_k  M^j_{i+1}\ar[d]^-{1\otimes d}\ar[r]^-{\del}&\cdots\\
\cdots\ar[r]^-{\del}&\Lambda^\ast\otimes_k  M^{j+1}_i\ar[r]^-{\del}\ar[d]^-{1\otimes d}&     \Lambda^\ast\otimes_k M^{j+1}_{i+1}\ar[r]^-{\del}\ar[d]^-{1\otimes d}&\cdots\\
&\vdots&\vdots&
}
\end{equation}
where the $l$-th component of $\Phi(M)$ is $\coprod_{i+j=l}\Lambda^\ast\otimes_k M_i^j$. 
\end{chunk}

\begin{chunk}\label{classical BGG}
Keep the notation as above. 
Since $\Phi$ preserves homotopy, suspensions and mapping cones, it induces an exact functor $\Phi\colon \K(S\text{-}\Gr)\rightarrow \K(\Lambda\text{-}\Gr)$. The image of this functor lies in $\K(\Lambda\text{-}\GrInj)$ because $\Lambda^\ast$ is graded-injective.
Bernstein, Gel'fand and Gel'fand \cite[Theorem 3]{BGG} proved that $\Phi$ naturally induces a triangle equivalence
$$
\Phi\colon \D^b(S\text{-}\gr)\stackrel{\sim}\longrightarrow \D^b(\Lambda\text{-}\gr);
$$
see also \cite[Theorem 2.12.1]{BGS}.
This is known as the BGG correspondence. Moreover, it fits into the following commutative diagram
\begin{equation}\label{lifting}
     \xymatrix{
    \D^b(S\text{-}\gr)\ar[r]^-{\sim}\ar[d]_-{\mathrm{inc}}& \D^b(\Lambda\text{-}\gr )\ar[d]^-{\mathsf{i}}\\
    \D(S\text{-}\Gr)\ar[r]& \K(\Lambda\text{-}\GrInj),
    }
\end{equation}
where the bottom map is the following composition
$$
\D(S\text{-}\Gr)\stackrel{\mathsf{p}}\longrightarrow  \K(S\text{-}\Gr)\stackrel{\Phi}\longrightarrow \K(\Lambda\text{-}\GrInj),
$$
here $\mathsf{p}$ is the left adjoint of the localization functor $\K(S\text{-}\Gr)\rightarrow \D(S\text{-}\Gr)$ (see \cite[Proposition 2.12]{BN} for its exisence).

The essential images of the vertical functors in (\ref{lifting}) are precisely the full subcategories of compact objects in the bottom categories. 
This is clear for the left one as the global dimension of $S$ is finite. See \ref{compact in K(GrInj)} for the right one.
Combine with that $\Phi\circ \mathsf{p}$ preserves coproducts, Lemma \ref{test equivalence} yields $\Phi\circ \mathsf{p}$ is an equivalence; see \cite[Example 5.7]{Krause}.
\end{chunk}
Now we define the exact functor $ \D_n(S\text{-}\Gr)\rightarrow \K_n(\Lambda\text{-}\GrInj).$
\begin{chunk}
For a $n$-periodic complex $M\in \C_n(S\text{-}\Gr)$, the total complex $\Phi(M)$ (see (\ref{total complex})) is a $n$-periodic complex in $\C_n(\Lambda\text{-}\Gr)$. Therefore this gives a  functor
$$
\Phi^\prime\colon \C_n(S\text{-}\Gr)\longrightarrow \C_n(\Lambda\text{-} \Gr)
$$
which maps $M$ to $\Phi(M)$.
Also, $\Phi^\prime$ induces an exact functor $\Phi^\prime\colon \K_n(S\text{-}\Gr)\rightarrow \K_n(\Lambda\text{-}\Gr)$ between the homotopy categories and its image lies in $\K_n(\Lambda\text{-}\GrInj)$. Consider the following composition
$$
\D_n(S\text{-}\Gr)\stackrel{\mathsf{p}^\prime}\longrightarrow \K_n(S\text{-} \Gr)\stackrel{\Phi^\prime}\longrightarrow  \K_n(\Lambda\text{-}\GrInj),
$$
where $\mathsf{p}^\prime$ is the left adjoint of the localization functor $\K_n(S\text{-}\Gr)\rightarrow \D_n(S\text{-}\Gr)$; its existence can refer the non-graded version of Remark \ref{homotopy projective}.
\end{chunk}

\begin{proof}[Proof of Theorem \ref{Koszul duality}]
It follows from Proposition \ref{embedding} and \ref{compact in K(GrInj)} that $\D_n(S\text{-}\Gr)$ and $\K_n(\Lambda\text{-}\GrInj)$ are compactly generated triangulated categories.
Combine with \ref{classical BGG}, we observe that there exists a commutative diagram
$$
\xymatrix{
\D(S\text{-}\Gr)\ar[r]^-{\Phi\circ \mathsf{p}}_-{\sim} \ar[d]_-{\Delta} & \K(\Lambda\text{-}\GrInj )\ar[d]^-{\Delta}\\
\D_n(S\text{-}\Gr)\ar[r]^-{\Phi^\prime\circ \mathsf{p}^\prime}& \K_n( \Lambda\text{-}\GrInj).
}
$$
Since $\Phi^\prime\circ \mathsf{p}^\prime$ preserves coproducts, Proposition \ref{embedding} and Lemma \ref{induce equivalence} imply $\Phi^\prime\circ \mathsf{p}^\prime$ is a triangle equivalence.
\end{proof}

\bibliographystyle{amsplain}
\bibliography{ref}
\end{document}